\documentclass[preprint,10pt]{elsarticle}
\usepackage{times}
\usepackage{amsthm}
\usepackage{amssymb}
\usepackage[all]{xy}
\usepackage[mathscr]{eucal}
\usepackage{amsmath,latexsym,oldlfont}
\usepackage{mathdots}
\usepackage{pifont}
\usepackage{stmaryrd}
\usepackage{textcomp}
\usepackage{pifont}
\usepackage{multirow}
\usepackage[colorlinks,linkcolor=blue,citecolor=red,linktocpage=true]{hyperref}
\usepackage{mathtools}
\numberwithin{equation}{section}
\newtheorem{thm}{Theorem}[section]
\newtheorem{cor}[thm]{Corollary}
\newtheorem{lem}[thm]{Lemma}
\newtheorem{prop}[thm]{Proposition}

\theoremstyle{definition}

\theoremstyle{remark}
\newtheorem{rem}{Remark}[section]

\newtheorem{example}{Example}[section]
\newcommand{\swedge}{{\scriptstyle\wedge}}
\newcommand{\dual}{{\scriptscriptstyle\vee}}
\newcommand{\nwedge}{\mathchoice{{\textstyle\wedge}}%
    {{\wedge}}%
    {{\textstyle\wedge}}%
    {{\scriptstyle\wedge}}}
\newsavebox{\spacebox}
\begin{lrbox}{\spacebox}
	\verb*! !
\end{lrbox}
\newcommand{\aspace}{\usebox{\spacebox}}
\let\oldaspace\aspace
\renewcommand{\aspace}[1][0pt]{%
	\mathrel{\raisebox{#1}{$\oldaspace$}}%
}
\journal{Linear Algebra and its Applications}

\begin{document}

\begin{frontmatter}

\title{An explicit description in terms of Pl\"ucker coordinates of the Lagrangian-Grassmannian}
\author[uacm]{Jes\'us Carrillo--Pacheco\corref{cor1}}
\ead{jesus.carrillo@uacm.edu.mx}
\author[uacm]{Fausto Jarqu\'\i n--Z\'arate}
\ead{fausto.jarquin@uacm.edu.mx}
\author[uacm]{Maurilio Velasco--Fuentes}
\ead{maurilio.velasco.fuentes@uacm.edu.mx}
\author[uam]{Felipe Zald\'{\i}var}
\ead{fz@xanum.uam.mx}
\address[uacm]{Academia de Matem\'aticas, Universidad Aut\'onoma de la Ciudad de M\'exico, 09790 M\'exico, D.~F., 
M\'exico.\tnoteref{t1}}
\tnotetext[t1]{
J. Carrillo-Pacheco, F. Jarqu{\'\i}n-Z\'arate and
M. Velasco-Fuentes were supported by the Cryptography
Laboratory Project PI2013-29. Secretar\'{\i}a de Ciencia,
Tecnolog\'{\i}a e Innovaci\'on del Distrito Federal, M\'exico.}
\address[uam]{Departamento de Matem\'aticas, Universidad Aut\'onoma Metropolitana-I, 09340 M\'exico, D.~F., M\'exico.}
\cortext[cor1]{Corresponding author}

\begin{abstract}
For an arbitrary field of any characteristic
we give an explicit description, in terms of Pl\"ucker coordinates,  of the projective linear space
that cuts out the Lagrangian-Grassmannian variety $L(n,2n)$ of
maximal isotropic subspaces in a symplectic vector space  of dimension $2n$ in the
Grassmannian variety $G(n,2n)$.
\end{abstract}


\end{frontmatter}

\section{Introduction}\label{sec:sec1}
Let $E$ be  a finite dimensional symplectic vector space over an
arbitrary field $F$, with symplectic form $\langle\; ,\;\rangle$.
Thus, $E$ has even dimension, say $2n$. Recall that a vector
subspace $W\subseteq E$ is {\it isotropic} if $\langle
x,y\rangle=0$ for all $x,y\in W$, and if $W$ is isotropic its
dimension is at most $n$. The {\it Lagrangian-Grassmannian}
$L(n,2n)$ is the projective variety given by the isotropic vector
subspaces  $W\subseteq E$ of maximal dimension $n$:
$$L(n,2n)=\{W\in G(n,2n): W\;\text{is isotropic}\},$$
where $G(n,2n)$ denotes the Grassmannian variety of vector subspaces of dimension $n$ of $E$.
Denote by $\nwedge^rE$ the $r$-th exterior power of $E$.
The {\it Pl\"ucker embedding} is the regular map
$G(n,2n)\rightarrow {\mathbb P}(\wedge^nE)$ given on each $W\in G(n,2n)$
by choosing first a basis $w_1,\ldots, w_n$ of $W$ and then mapping the vector subspace
$W\in G(n,2n)$ to the tensor $w_1\swedge\cdots\swedge w_n\in \nwedge^nE$.
Since choosing a different basis  of $W$ changes the tensor $w_1\swedge\cdots\swedge w_n$
by a nonzero scalar, this tensor is a well-defined element in the projective space
${\mathbb P}(\nwedge^nE)\simeq{\mathbb P}^{N-1}$, where $N=\binom{2n}{n}=\dim_F(\nwedge^nE)$.
Under the Pl\"ucker embedding, the Lagrangian-Grassmannian is given by
$$L(n,2n)=\{w_1\swedge\cdots\swedge w_n\in G(n,2n): \langle w_i,w_j\rangle=0\;\text{for all $1\leq i<j\leq n$}\}.$$

Now, by wedging with the symplectic form $\langle\; ,\;\rangle$
viewed as a $2$-tensor in $\nwedge^2E^{\dual}$, where $E^{\dual}$
is the dual vector space, we have the contraction map
$$f:\nwedge^nE\rightarrow \nwedge^{n-2}E$$
given by
$$f(w_1\swedge\cdots\swedge w_n)=\sum_{1\leq s<t\leq n}\langle w_s,w_t\rangle
w_1\swedge\cdots\swedge \widehat{w}_s\swedge\cdots\swedge \widehat{w}_t\swedge\cdots\swedge w_n,$$
where $\widehat{w}$ means that the corresponding term is omitted.
If ${\mathbb P}(\ker f)$ is the projectivization of $\ker f$, in \cite{1} it is proved that
$$L(n,2n)=G(n,2n)\cap{\mathbb P}(\ker f).$$

Now, choose a basis $\{e_1,\ldots,e_{2n}\}$ of the  symplectic space $E$ such that
$$
\langle e_i,e_j\rangle=\begin{cases}
1& \text{if $j=2n-i+1$}, \\
0 & \text{otherwise},
\end{cases}$$
and define the set $I(d,2n)=\{\alpha=(\alpha_1,\ldots,\alpha_d):
1\leq\alpha_1<\cdots<\alpha_d\leq 2n\}$. Then, for $\alpha=(\alpha_1,\ldots,\alpha_n)\in I(n,2n)$ write
\begin{align*}
e_{\alpha}&:=e_{\alpha_1}\swedge\cdots\swedge e_{\alpha_n},\\
e_{\alpha_{st}}&:=e_{\alpha_1}\swedge\cdots\swedge\widehat{e}_{\alpha_s}
\swedge\cdots\swedge\widehat{e}_{\alpha_t}\swedge\cdots   \swedge e_{\alpha_n},
\end{align*}
and
$$p_{i\alpha_{st}(2n-i+1)}:=p_{i\alpha_1 \cdots \widehat{\alpha}_s
\cdots \widehat{\alpha}_t \cdots \alpha_n(2n-i+1)},$$ for the
corresponding Pl\"ucker coordinate. Then, in \cite[Proposition
6]{1} the kernel of $f$ is characterized as follows: For $w\in
\nwedge^nE$ written in Pl\"ucker coordinates as
$$w=\sum_{\alpha\in I(n,2n)}p_{\alpha}e_{\alpha}$$
we have that
\begin{equation}\label{eq2.1}
w\in\ker f\iff \sum_{i=1}^np_{i\alpha_{st}(2n-i+1)}=0,\;\text{for all $\alpha_{st}\in I(n-2,2n)$}.
\end{equation}

In \cite[Section 3]{1} these linear forms were given the following description:
For $\alpha_{st}\in I(n-2,2n)$ define the linear polynomials
$$\Pi_{\alpha_{st}}:=\sum_{i=1}^nc_{i,\alpha_{st},2n-i+1}X_{i,\alpha_{st},2n-i+1},$$
with
$$c_{i,\alpha_{st},2n-i+1}=\begin{cases}
1  & \text{if $|\text{supp}\{i,\alpha_{st},2n-i+1\}|=n$}, \\
0 & \text{otherwise},
\end{cases}$$
where $\text{supp}(\beta)=\{\beta_1,\ldots,\beta_d\}$ for
$\beta=(\beta_1,\ldots,\beta_d)\in I(d,2n)$. With this notation,
the generators of $\ker f$ are the polynomials $\Pi_{\alpha_{st}}$,
for $\alpha_{rs}\in I(n-2,2n)$.  If $Q_{\gamma}$ denote the quadratic
Pl\"ucker polynomials that define the Grassmann variety $G(n,2n)$ in ${\mathbb P}^{N}$, in terms of the linear forms
$\Pi_{\alpha_{st}}$, by  \cite[Section 3]{1}
we have the following characterization of $L(n,2n)$,
as the common zeros of the quadratic polynomials $Q_{\gamma}$
and the linear polynomials $\Pi_{\alpha_{st}}$, that is
\begin{equation}\label{eq2.2}
L(n,2n)={\EuScript Z}\langle Q_{\gamma}; \Pi_{\alpha_{st}}: \alpha_{st}\in I(n-2,2n)\rangle.
\end{equation}

The main result of this paper, established in
Theorem~\ref{thmPeven} and Theorem~\ref{thmPodd}, is to obtain an
explicit description in terms of Pl\"ucker coordinates  of the
linear space ${\mathbb P}(\ker f)$ over an arbitrary field $F$ of
any characteristic. Explicitly, we give a characterization of the
homogeneous linear system of equations that define ${\mathbb
P}(\ker f)$ for any positive integer $n$. The analysis is divided
in two parts: for $n$ even and then for $n$ odd; the odd case will be deducted
from the ideas of the even case. This explicit result is useful when
one studies the linear code associated to the
Lagrangian-Grassmannian as in~\cite{1}.
The following formula in Pl\"ucker coordinates will be used throughout this paper and thereafter
we  call it a {\it  Pl\"ucker linear relation}:
 $$X_{1\aspace[1.0pt] 2n}+X_{2\aspace[1.0pt] (2n-1)}+\cdots+X_{n\aspace[1.0pt] (n+1)}=0,$$
 where the space symbols $\aspace[1.5pt]$ should be replaced by elements $\alpha_{st}\in I(n-2,2n)$ in such a
 way that we obtain homogeneous linear equations
$$\Pi_{\alpha_{st}}:=X_{1,\alpha_{st},  2n}+X_{2,\alpha_{st},  (2n-1)}+\cdots+X_{n,\alpha_{st},  (n+1)}=0$$
in $k$-variables, where the term $X_{i,\alpha_{st},  (2n-i+1)}$ does
not appear if $|\text{supp}\{i,\alpha_{st},  (2n-i+1)\}|<n$. When
this happens we say $\Pi_{\alpha_{st}}$ is a  $k$-plane, and we
obtain a  combinatorial characterization of $\Pi_{\alpha_{st}}$ in
Lemma \ref{lem2.3}.

\section{A combinatorial description of the kernel of the contraction map}\label{sec:sec2}

For the system of linear homogeneous equations
$\Pi_{\alpha_{st}}=0$, $\alpha_{st}\in I(n-2,2n)$, that define the
kernel of the contraction map $f:\nwedge^nE\rightarrow
\nwedge^{n-2}E$, we describe its associated matrix $B$, see
\cite[Section 3]{1}, in terms of a combinatorial construction
using a family of matrices ${\EuScript L}_k$. As usual, let $I_s$
denote the identity $s\times s$ matrix, and we will sometimes use
this notation when $s$ is a binomial coefficient, that we denote
by $C^a_b=\binom{a}{b}$.

\subsection{An iterative process to construct a family of matrices}
For any integer $s\geq 1$, write {\boldmath$s$} $=(1,\ldots,1)$ a row matrix with $s$
entries equal to $1$. For any integer $k\geq 1$ define the $(k+1)\times C^{k+1}_{2}$ matrix $A^0_k$ as
$$A_k^0 =\left( \begin{tabular}{ c c  c c c c c}
\cline{1-1}\multicolumn{1}{|c|}{\boldmath${k}$}  & & & & &    \\
\cline{2-1}\cline{1-1} \multicolumn{1}{|c}{}  &
\multicolumn{1}{|c|}{\boldmath${k-1}$} & & & & \\
\cline{2-2}\multicolumn{1}{|c}{} & \multicolumn{1}{|c|}{} &
\multicolumn{1}{c}{$\ddots$} & & & \\
\cline{4-2}\multicolumn{1}{|c}{} &\multicolumn{1}{|c|}{} & \multicolumn{1}{c|}{$\cdots$}&
\multicolumn{1}{c|}{\boldmath${3}$}  & & \\
\cline{5-2}\cline{4-4}\multicolumn{1}{|c}{} & \multicolumn{1}{|c|}{} &
\multicolumn{1}{c|}{$\cdots$}& \multicolumn{1}{c|}{} &
\multicolumn{1}{c|}{\boldmath${2}$}   & \\
\cline{6-2}\cline{5-5} \multicolumn{1}{|c}{} & \multicolumn{1}{|c|}{} &
\multicolumn{1}{c|}{$\cdots$}& \multicolumn{1}{c|}{} &
\multicolumn{1}{c|}{} & \multicolumn{1}{c|}{\boldmath${1}$}\\
\cline{6-6}\multicolumn{1}{|c}{$I_k$} & \multicolumn{1}{|c|}{$I_{k-1}$} &
\multicolumn{1}{c|}{$\cdots$}& \multicolumn{1}{c|}{$I_3$} &
\multicolumn{1}{c|}{$I_2$} & \multicolumn{1}{c|}{$I_1$}    \\
\hline
\end{tabular}\right),$$
where $I_k, I_{k-1},\ldots,  I_3, I_2,I_1$ are the corresponding identity matrices and the non-marked spaces on the upper right blocks are filled-in with zeroes. Now consider the following $\big(1+k+C^{k+1}_{2}\big)\times C^{k+1}_{2}$ matrix
given by placing an identity matrix of size $C^{k+1}_{2}$ at the bottom part of the above $A^0_k$ matrix, that is
the matrix
$$A_k^0(I_{C^{k+1}_2}) =\left( \begin{tabular}{ c c  c c c c c}
\cline{1-1}\multicolumn{1}{|c|}{\boldmath${k}$}  & & & & &    \\
\cline{2-1}\cline{1-1} \multicolumn{1}{|c}{}  &
\multicolumn{1}{|c|}{\boldmath${k-1}$} & & & & \\
\cline{2-2}\multicolumn{1}{|c}{} & \multicolumn{1}{|c|}{} &
\multicolumn{1}{c}{$\ddots$} & & & \\
\cline{4-2}\multicolumn{1}{|c}{} &\multicolumn{1}{|c|}{} & \multicolumn{1}{c|}{$\cdots$}&
\multicolumn{1}{c|}{\boldmath${3}$}  & & \\
\cline{5-2}\cline{4-4}\multicolumn{1}{|c}{} & \multicolumn{1}{|c|}{} &
\multicolumn{1}{c|}{$\cdots$}& \multicolumn{1}{c|}{} &
\multicolumn{1}{c|}{\boldmath${2}$}   & \\
\cline{6-2}\cline{5-5} \multicolumn{1}{|c}{} & \multicolumn{1}{|c|}{} &
\multicolumn{1}{c|}{$\cdots$}& \multicolumn{1}{c|}{} &
\multicolumn{1}{c|}{} & \multicolumn{1}{c|}{\boldmath${1}$}\\
\cline{6-6}\multicolumn{1}{|c}{$I_k$} & \multicolumn{1}{|c|}{$I_{k-1}$} &
\multicolumn{1}{c|}{$\cdots$}& \multicolumn{1}{c|}{$I_3$} &
\multicolumn{1}{c|}{$I_2$} & \multicolumn{1}{c|}{$I_1$}    \\
\hline
\multicolumn{1}{|c}{}& & & & & \multicolumn{1}{c|}{}\\
\multicolumn{1}{|c}{}& & & & & \multicolumn{1}{c|}{}\\
\multicolumn{1}{|c}{}& & \multicolumn{1}{c}{$I_{C^{k+1}_2}$}& & & \multicolumn{1}{c|}{}\\
\multicolumn{1}{|c}{}& & & & & \multicolumn{1}{c|}{}\\ \hline
\end{tabular}\right).$$
Next, define the block-stepped matrix $A^1_k$ by joining the
above matrices $A_k^0(I_{C^{k+1}_2})$ side-by-side and aligning the
bottoms of the corresponding identity matrices, that is
$$A^1_k=A_k^0(I_{C^{k+1}_2})\sqcup A_{k-1}^0(I_{C^{k}_2})\sqcup\cdots\sqcup
A_2^0(I_{C^{3}_2})\sqcup A_1^0(I_{C^{2}_2}),$$
where $\sqcup$ means
joining together side-by-side and aligning the bottoms of the
corresponding identity matrices and filling the non-marked spaces on
the upper right blocks with zeroes. Notice that this is a matrix of
size
$$\big(k+1+C_2^{k+1}\big)\times \big(C_2^{k+1}+C_2^{k}+\cdots+C_2^{3}+C_2^{2}\big)=
\big(k+1+C_2^{k+1}\big)\times C_3^{k+2}.$$
Then, define the matrix $A^1_k(I_{C_3^{k+2}})$
obtained by adding the identity matrix $I_{C_3^{k+2}}$ to the bottom of
$A^1_k$. Thus, $A^1_k(I_{C_3^{k+2}})$ is  matrix of size $\big(k+1+C_2^{k+1}+C_3^{k+2}\big)\times C_3^{k+2}$.

We iterate these constructions to obtain matrices
$A^2_k,A^3_k,\ldots, A^m_k,\ldots$ and $A^2_k(I_{C_4^{k+3}})$,
$A^3_k(I_{C_5^{k+4}})$, $\ldots,
 A^m_k(I_{C_{m+2}^{k+m+1}})$, etcetera. Explicitly, for any integer
 $\ell\geq 0$ the matrix $A_k^{\ell+1}$ is given by
$$\label{iteraMatriz}
 A_k^{\ell+1}=A_k^{\ell}\left( I_{C^{k+\ell+1}_{\ell+2}} \right)\sqcup
A_{(k-1)}^{\ell}\left( I_{C^{k+\ell}_{ \ell+2}} \right)\sqcup
A_{(k-2)}^{\ell}\left( I_{C^{k+\ell-1}_{\ell+2}} \right) \sqcup
 \cdots \sqcup A_2^{\ell}\left( I_{C^{\ell+3}_{\ell+2}} \right)
\sqcup A_1^{\ell}\left( I_{C^{\ell+2}_{\ell+2}} \right),
$$
where $\sqcup$ is as before.

\subsection{The matrices ${\EuScript L}_k$}

Assume  that $n\geq 4$ is an \emph{even} integer and consider all even
integers $m$ such that $4\leq m\leq n$. Write $k=(m+2)/2$, and let
${\EuScript L}_k$ be the matrix
$${\EuScript L}_k=A_k^{k-3},$$
where we observe that its number of rows is
$$1+C^k_1+C^{k+1}_{2}+\cdots+C^{2k-4}_{k-3}+C^{2k-3}_{k-2}=C^{2k-2}_{k-2}=C^m_{(m-2)/2},$$
and its number of columns is
$$C^{2k-3}_{k-2}+C^{2k-4}_{k-2}+\cdots+C^{k-1}_{k-2}+C^{k-2}_{k-2}=C^{2k-2}_{k-1}=C^m_{m/2}.$$
Moreover, the matrix ${\EuScript L}_k$ has $k=(m+2)/2$ ones in each row,
and has $k-1$ ones in each column.

\begin{example}
In Example 4 of \cite{1}, for the contraction map
$f:\nwedge^4E\rightarrow \nwedge^{2}E$ we obtained that its kernel
is given by the solutions to a couple of systems of equations, one of them consisting of twenty-four $2$-planes and the other one consisting of the following four $3$-planes
\begin{align*}
\Pi_{1,8}:\quad& X_{1278}+X_{1368}+X_{1458}=0 \\
\Pi_{2,7}:\quad&X_{1278}+X_{2367}+X_{2457}=0 \\
\Pi_{3,6}:\quad&X_{1368}+X_{2367}+X_{3456}=0 \\
\Pi_{4,5}:\quad&X_{1458}+X_{2457}+X_{3456}=0,
\end{align*}
whose associated matrix ${\EuScript L}_3$ is:
$$ {\EuScript L}_3 = A_3^0 = \begin{pmatrix} 1&1&1&0&0&0\\
                                       1& 0 &0&1&1&0 \\
                                       0 & 1& 0 & 1 & 0 & 1 \\
                                        0 & 0 & 1 & 0 & 1 & 1  \end{pmatrix}.$$
\end{example}

\begin{example}
Let $m=10$ and $k=6$. By definition ${\EuScript L}_6 = A_6^3$,
where:

$$ A_6^0= \begin{pmatrix}\begin{tabular}{|c c c c c c }
\cline{1-1}\multicolumn{1}{|c|}{\boldmath${6}$}  & & & & &    \\ 
\cline{2-1}\cline{1-1} \multicolumn{1}{|c}{}  &
\multicolumn{1}{|c|}{\boldmath${5}$} & & & & \\ 
\cline{3-2}\cline{2-2}\multicolumn{1}{|c}{} & \multicolumn{1}{|c|}{} &
\multicolumn{1}{c|}{\boldmath${4}$} & & & \\ 
\cline{3-3}\cline{4-2}\multicolumn{1}{|c}{} &\multicolumn{1}{|c|}{} & \multicolumn{1}{c|}{}&\multicolumn{1}{c|}{\boldmath${3}$}  & & \\
\cline{5-2}\cline{4-4} \multicolumn{1}{|c}{} & \multicolumn{1}{|c|}{} &
\multicolumn{1}{c|}{}& \multicolumn{1}{c|}{} &
\multicolumn{1}{c|}{\boldmath${2}$}   & \\
\cline{6-2}\cline{5-5} \multicolumn{1}{|c}{} & \multicolumn{1}{|c|}{} &
\multicolumn{1}{c|}{}& \multicolumn{1}{c|}{} & \multicolumn{1}{c|}{}
& \multicolumn{1}{c|}{\boldmath${1}$}
\\ \cline{6-6}
\multicolumn{1}{|c}{$I_6$} & \multicolumn{1}{|c|}{$I_5$} &
\multicolumn{1}{c|}{$I_4$}& \multicolumn{1}{c|}{$I_3$} &
\multicolumn{1}{c|}{$I_2$} & \multicolumn{1}{c|}{$I_1$}    \\ \hline
\end{tabular} \end{pmatrix}
$$

\begin{align*}
A_6^1&=A_6^0 \left(I_{C^7_2}
\right)\sqcup A_5^0 \left( I_{C^6_2} \right)\sqcup \cdots \sqcup
A_2^0\left( I_{C^3_2} \right)\sqcup A_1^0 \left( I_{C^2_2} \right)\\
A_6^2&=A_6^1\left( I_{C^8_3}
\right) \sqcup A_5^1 \left( I_{C^7_3} \right) \sqcup \cdots \sqcup
A_2^1\left( I_{C^4_3} \right) \sqcup A_1^1\left( I_{C^3_3} \right)\\
A_6^3&=A_6^2\left( I_{C^9_4}
\right) \sqcup A_5^2 \left( I_{C^8_4} \right)\sqcup \cdots \sqcup
A_2^2\left( I_{C^5_4} \right)\sqcup A_1^2 \left( I_{C^4_4}\right).
\end{align*}
\end{example}

\subsection{A combinatorial description of the matrix $B$ associated to the contraction map}

Let $m\in{\mathbb N}$ be an even integer. Define
\begin{align*}
P_i&=(i,2m-i+1),\quad\text{for $1\leq i\leq m$}\\
\Sigma_s&=\{P_1,\ldots,P_s\},\quad\text{for $1\leq s\leq m$}\\
\Sigma_{[s,m]}&=\{P_{s+1},\ldots,P_m\}.
\end{align*}
Throughout this paper, suppose $m\geq 8$ and denote by
$C_{(m-2)/2}(\Sigma_m)$ the family of subsets of $\Sigma_m$ with
$(m-2)/2$ elements, and similarly for $C_{(m-6)/2}(\Sigma_{m-2})$.
For $\alpha=(\alpha(1),\ldots,\alpha((m-6)/2))\in I((m-6)/2,m-2)$,
let
$$T_{(\alpha(1),\ldots,\alpha((m-6)/2))}:=(P_{\alpha(1)},
\ldots, P_{\alpha((m-6)/2)})\times
C_2\big(\Sigma_{[\alpha((m-6)/2),m]}\big).$$
We call these sets
{\it triangular arrays } or {\it triangles}.
 The family
$T_{(\alpha(1),\ldots,\alpha((m-6)/2))}$, varying
$(\alpha(1),\ldots,\alpha((m-6)/2))\in I((m-6)/2,m-2)$ has the
Bruhat order \cite{5.1}, that is,
$$T_{(\alpha(1),\ldots,\alpha((m-6)/2))}\leq T_{(\beta(1),\ldots,\beta((m-6)/2))}$$
if and only if
$$(\alpha(1),\ldots,\alpha((m-6)/2))\leq (\beta(1),\ldots,\beta((m-6)/2)).$$
\begin{lem}\label{lem2.1}
Let $m\geq8$ an even integer. Then
$$ C_{(m-2)/2}(\Sigma_m)=\bigcup_{\alpha\in I((m-6)/2,m-2)}T_{\alpha}.$$
\end{lem}
\begin{proof}
By  construction, each $T_{\alpha}\subseteq C_{(m-2)/2}(\Sigma_m)$. Now, let
$$\{P_{\alpha(1)},\ldots,P_{\alpha((m-2)/2)}\}\in C_{(m-2)/2}(\Sigma_m).$$
Then, $(\alpha(1),\ldots,\alpha((m-2)/2))\in I((m-2)/2, m)$ and hence $(\alpha(1),\ldots,\alpha((m-6)/2))\in I((m-6)/2, m-2)$ and $(\alpha((m-4)/2),\alpha((m-2)/2))$ satisfy that
$$\alpha((m-6)/2)+1\leq \alpha((m-4)/2)<\alpha((m-2)/2)\leq m.$$
It follows that
\begin{align*}
\{P_{\alpha(1)},\ldots,P_{\alpha((m-2)/2)}\}& \in (P_{\alpha(1)},\ldots,P_{\alpha((m-6)/2)})\times C_2(\Sigma_{[\alpha((m-6)/2)+1,m]})\\
&\quad =T_{(\alpha(1),\ldots,\alpha((m-6)/2))}.
\end{align*}
\end{proof}

\begin{example}
For $m=6$ we just have one triangle
$$ T = \left\{ \begin{matrix} P_1P_2 & P_1P_3 & \cdots & P_1P_6 \\
                              \quad & P_2P_3 & \cdots & P_2P_6 \\
                               \quad & \quad & \ddots & \vdots \\

                                \quad & \quad & \quad & P_5P_6 \end{matrix}
                                \right\}.$$
Note that $C_2(\Sigma_6) = T$, as in Lemma \ref{lem2.1}.
\end{example}

\begin{example}
For $m=8$ we have six triangles, namely
$$ T_i = P_i \times C_2(\Sigma_{[i,8]}) \text { with } i=1,2,\ldots, 6,$$
where $P_i$ is as above. Note that $C_3(\Sigma_8) = \cup_{i=1}^6
T_i$, as in Lemma \ref{lem2.1}.
\end{example}

\smallskip
\subsection{Construction of an auxiliary matrix $M_m$}\label{Sec-Mm}
Now, let $\varphi:C_{(m-2)/2}(\Sigma_m)\rightarrow F^{C^{2m}_m}$ be the
function given by
$$(P_{\alpha(1)},\ldots,P_{\alpha((m-2)/2)})\mapsto \big( \varphi^{\beta(1),\ldots,\beta(m/2)}_{P_{\alpha(1)},\ldots, P_{\alpha((m-2)/2)},P_i}\big)_{\beta\in I(m/2,m),\ }$$
where $$P_i\in \Sigma_m-\{P_{\alpha(1)},\ldots,P_{\alpha((m-2)/2)}
\},\quad i=1,\ldots,m,$$ and
$$\varphi^{\beta(1),\ldots,\beta(m/2)}_{P_{\alpha(1)},\ldots, P_{\alpha((m-2)/2)},P_i}=\begin{cases}
1  & \text{if $\big(\alpha(1),\ldots,\alpha((m-2)/2),i\big)=\big(\beta(1),\ldots,\beta(m/2)\big)$}, \\
0& \text{otherwise},
\end{cases}$$
thus, $\varphi$ defines a row vector of weight $r=\frac{m+2}{2}$,
whose columns are labeled by $\beta$. Varying $\beta$, define the
corresponding $C^m_{(m-2)/2}\times C^m_{m/2}$ matrix as
\begin{align}\label{matrizM}
M_m=\begin{pmatrix}
 \varphi^{\beta(1),\ldots,\beta(m/2)}_{P_{\alpha(1)},\ldots, P_{\alpha((m-2)/2)},P_i}
\end{pmatrix}
\end{align}
for $\alpha\in I((m-2)/2,m)$, $\beta\in I(m/2,m)$, and
$i=1,\ldots,m$.

\begin{example} The matrix $M_4$.
In this case $m=4$, $\Sigma_4 = \{P_1,P_2,P_3,P_4\}$ and
$\varphi:C_1(\Sigma_4) \longrightarrow F^{C^4_2}$. It is easy to see that
\begin{align*}
\varphi(P_1) =
\underbracket{\;\;1\;\;}_{12},\underbracket{\;\;1\;\;}_{13},\underbracket{\;\;1\;\;}_{14},
\underbracket{\;\;0\;\;}_{23},\underbracket{\;\;0\;\;}_{24},
\underbracket{\;\;0\;\;}_{34}, \\
\varphi(P_2) =
\underbracket{\;\;1\;\;}_{12},\underbracket{\;\;0\;\;}_{13},\underbracket{\;\;0\;\;}_{14},
\underbracket{\;\;1\;\;}_{23},\underbracket{\;\;1\;\;}_{24},
\underbracket{\;\;0\;\;}_{34}, \\
\varphi(P_3) =
\underbracket{\;\;0\;\;}_{12},\underbracket{\;\;1\;\;}_{13},\underbracket{\;\;0\;\;}_{14},
\underbracket{\;\;1\;\;}_{23},\underbracket{\;\;0\;\;}_{24},
\underbracket{\;\;1\;\;}_{34}, \\
\varphi(P_4) =
\underbracket{\;\;0\;\;}_{12},\underbracket{\;\;0\;\;}_{13},\underbracket{\;\;1\;\;}_{14},
\underbracket{\;\;0\;\;}_{23},\underbracket{\;\;1\;\;}_{24},
\underbracket{\;\;1\;\;}_{34}
\end{align*}
Hence, up to permutations of rows, we obtain the matrix
$$ M_4 = \begin{pmatrix}
1 & 1 & 1 & 0 & 0 & 0 \\
1 & 0 & 0 & 1 & 1 & 0 \\
0 & 1 & 0 & 1 & 0 & 1 \\
0 & 0 & 1 & 0 & 1 & 1 \\
\end{pmatrix}.$$
\end{example}

\begin{example}\label{EjM6} The matrix $M_6$.
In this case, $m=6$, $\Sigma_6 = \{P_1,P_2,P_3,P_4,P_5,P_6\}$, and
$\varphi:C_2(\Sigma_6) \longrightarrow F^{C_3^6}$, where
$C_2(\Sigma_6)$ is as above. Now, we have {\small \begin{align*}
\varphi(P_1P_2) & = (
\underbracket{\;\;1\;\;}_{123},\underbracket{\;\;1\;\;}_{124},\underbracket{\;\;1\;\;}_{125},
\underbracket{\;\;1\;\;}_{126},\underbracket{\;\;0\;\;}_{134},
\underbracket{\;\;0\;\;}_{135},
\underbracket{\;\;0\;\;}_{136},\underbracket{\;\;0\;\;}_{145},
\underbracket{\;\;0\;\;}_{146},\underbracket{\;\;0\;\;}_{156},\underbracket{\;\;0\;\;}_{234},\underbracket{\;\;0\;\;}_{235},
\ldots, \underbracket{\;\;0\;\;}_{356},\underbracket{\;\;0\;\;}_{456}),\\
\varphi(P_1P_3)  & = (
\underbracket{\;\;1\;\;}_{123},\underbracket{\;\;0\;\;}_{124},\underbracket{\;\;0\;\;}_{125},
\underbracket{\;\;0\;\;}_{126},\underbracket{\;\;1\;\;}_{134},
\underbracket{\;\;1\;\;}_{135},\underbracket{\;\;1\;\;}_{136},\underbracket{\;\;0\;\;}_{145},
\underbracket{\;\;0\;\;}_{146},
\underbracket{\;\;0\;\;}_{156},\underbracket{\;\;0\;\;}_{234},\underbracket{\;\;0\;\;}_{235},
\ldots, \underbracket{\;\;0\;\;}_{356},\underbracket{\;\;0\;\;}_{456}),\\
\varphi(P_1P_4) & = (
\underbracket{\;\;0\;\;}_{123},\underbracket{\;\;1\;\;}_{124},\underbracket{\;\;0\;\;}_{125},
\underbracket{\;\;0\;\;}_{126},\underbracket{\;\;1\;\;}_{134},
\underbracket{\;\;0\;\;}_{135},\underbracket{\;\;0\;\;}_{136},\underbracket{\;\;1\;\;}_{145},
\underbracket{\;\;1\;\;}_{146},\underbracket{\;\;0\;\;}_{156},\underbracket{\;\;0\;\;}_{234},\underbracket{\;\;0\;\;}_{235},
\ldots,
\underbracket{\;\;0\;\;}_{356},\underbracket{\;\;0\;\;}_{456}),\end{align*}}where
the subscripts under horizontal brackets are Pl\"ucker coordinates
in $\nwedge^6 E$. Continuing with the process for all $P_iP_j, 1\leq
i < j \leq 6$, we find that the matrix $M_6$ is given, up to
permutations of rows, by
$$
\left( \begin{tabular}{cccccccccccccccccccc}
1 & 1 & 1 & 1 & 0 & 0 & 0 & 0 & 0 & 0 & 0 & 0 & 0 & 0 & 0 & 0 & 0
& 0 & 0 & 0 \cr \cline{1-4}\multicolumn{1}{|c} 1 & 0 & 0 & 0 &
\multicolumn{1}{|c}{$1$} & 1 & 1 & 0 & 0 & 0 & 0 & 0 & 0 & 0 & 0 &
0 & 0 & 0 & 0 & 0 \cr \cline{5-7}\multicolumn{1}{|c} 0 & 1 & 0 & 0
& \multicolumn{1}{|c}{$1$} & 0 &  0 & \multicolumn{1}{|c} 1 & 1 &
0 & 0 & 0 & 0 & 0 & 0 & 0 & 0 & 0 & 0 & 0 \cr
\cline{8-9}\multicolumn{1}{|c} 0 & 0 & 1 & 0 & \multicolumn{1}{|c}
0 & 1 & 0 & \multicolumn{1}{|c} 1 & 0 & \multicolumn{1}{|c} 1 & 0
& 0 & 0 & 0 & 0 & 0 & 0 & 0 & 0 & 0 \cr
\cline{10-10}\multicolumn{1}{|c} 0 & 0 & 0 & 1 &
\multicolumn{1}{|c} 0 & 0 &  1 & \multicolumn{1}{|c} 0 & 1 &
\multicolumn{1}{|c} 1 & \multicolumn{1}{|c} 0 & 0 & 0 & 0 & 0 & 0
& 0 & 0 & 0 & 0 \cr \cline{1-10}\multicolumn{1}{|c} 1 & 0 & 0 & 0
& 0 & 0 & 0 & 0 & 0 & 0 & \multicolumn{1}{|c} 1 & 1 & 1 & 0 & 0 &
0 & 0 & 0 & 0 & 0 \cr \cline{11-13}\multicolumn{1}{|c} 0 & 1 & 0 &
0 & 0 & 0 & 0 & 0 & 0 & 0 &\multicolumn{1}{|c} 1 & 0 & 0 &
\multicolumn{1}{|c} 1 & 1 & 0 & 0 & 0 & 0 & 0 \cr
\cline{14-15}\multicolumn{1}{|c} 0 & 0 & 1 & 0 & 0 & 0 & 0 & 0 & 0
& 0 & \multicolumn{1}{|c} 0 & 1 & 0 & \multicolumn{1}{|c} 1 & 0 &
\multicolumn{1}{|c} 1 & 0 & 0 & 0 & 0 \cr \cline{16-16}
\multicolumn{1}{|c} 0 & 0 & 0 & 1 & 0 & 0 & 0 & 0 & 0 & 0 &
\multicolumn{1}{|c} 0 & 0 & 1 & \multicolumn{1}{|c} 0 & 1 &
\multicolumn{1}{|c} 1 & \multicolumn{1}{|c} 0 & 0 & 0 & 0 \cr
\cline{11-16}\multicolumn{1}{|c} 0 & 0 & 0 & 0 & 1 & 0 & 0 & 0 & 0
& 0 & \multicolumn{1}{|c} 1 & 0 & 0 & 0& 0 & 0 &
\multicolumn{1}{|c} 1 & 1 & 0 & 0 \cr \cline{17-18}
\multicolumn{1}{|c} 0 & 0 & 0 & 0 &  0 & 1 & 0 & 0 & 0 & 0
&\multicolumn{1}{|c}  0 & 1 & 0 & 0& 0 & 0 & \multicolumn{1}{|c} 1
& 0 & \multicolumn{1}{|c} 1 & 0 \cr
\cline{19-19}\multicolumn{1}{|c} 0 & 0 & 0 & 0 & 0 & 0 & 1 & 0 & 0
& 0 & \multicolumn{1}{|c} 0 & 0 & 1 & 0& 0 & 0 &
\multicolumn{1}{|c} 0 & 1 & \multicolumn{1}{|c} 1 &
\multicolumn{1}{|c} 0 \cr \cline{17-19}\multicolumn{1}{|c} 0 & 0 &
0 & 0 & 0 & 0 & 0 & 1 & 0 & 0 & \multicolumn{1}{|c} 0 & 0 &  0 &
1& 0 & 0 & \multicolumn{1}{|c} 1 & 0 & 0 & \multicolumn{1}{|c} 1
\cr \cline{20-20}\multicolumn{1}{|c} 0 & 0 & 0 & 0 & 0 & 0 & 0 & 0
& 1 & 0 & \multicolumn{1}{|c} 0 & 0 & 0 & 0& 1 & 0 &
\multicolumn{1}{|c} 0 & 1 & 0 & \multicolumn{1}{|c|} 1 \cr
\cline{20-20}\multicolumn{1}{|c} 0 & 0 & 0 & 0 & 0 & 0 & 0 & 0 & 0
& 1 & \multicolumn{1}{|c} 0 & 0 & 0 & 0& 0 & 1 &
\multicolumn{1}{|c} 0 & 0 & 1 & \multicolumn{1}{|c|} 1 \cr
\cline{1-20}
\end{tabular} \right).$$
\end{example}
\medskip

\begin{lem}\label{lem2.2}
Let $m \geq 10 $ be an even integer and $k=\frac{m+2}{2}$. Then, 
 $$\varphi(T_{(1,2,\ldots,\frac{m-8}{2},\frac{m-6}{2})})= A_k^1.$$
\end{lem}
\begin{proof}
Take $\displaystyle\alpha_0=(1,2,\ldots, (m-8)/2,(m-6)/2 )$ and
consider $$T_{\alpha_{0}}=P_{\alpha_{0}}\times
C_2(\Sigma_{[\frac{m-6}{2},m]})=(P_{({\alpha_0}\times
R_{\frac{m-4}{2}})}) \cup (P_{({\alpha_{0}}\times
R_{\frac{m-2}{2}})}) \cup \cdots \cup (P_{({\alpha_{0}}\times
R_{m-1})}),
$$
where $C_2(\Sigma_{[\frac{m-6}{2},m]})= R_{\frac{m-4}{2}} \cup R_{\frac{m-2}{2}} \cup\cdots \cup
R_{m-1}$, with
\begin{align*}
R_{\frac{m-4}{2}}&=\big\{ \textstyle  \big(\frac{m-4}{2},\frac{m-4}{2}+i\big): 1\leq i\leq \frac{m+4}{2}\big\},\\
R_{\frac{m-2}{2}}&= \big\{ \textstyle  \big(\frac{m-2}{2},\frac{m-2}{2}+i\big): 1\leq i\leq \frac{m+2}{2}\big\},\\
&\; \; \vdots \\
R_{m-1}&=\big\{ \textstyle  \big(m-1,m\big) \big\},
\end{align*}$P_{\alpha_0}= (P_1,P_2,\cdots, P_{\frac{m-8}{2}},
P_{\frac{m-6}{2}})$, and an analogous construction for $P_{(\alpha_0
\times R_{\frac{m-j}{2}})}$. If we now consider the corresponding
$P_{(\alpha_0\times (\frac{m-4}{2},\frac{m-4}{2}+i))}$, for
$R_{\frac{m-4}{2}}$, then the  first row of the matrix
$\varphi(T_{\alpha_0})$ is
$$\varphi(P_{(\alpha_0\times (\frac{m-4}{2},\frac{m-4}{2}+1))})=\varphi(P_{(\alpha_0\times
(\frac{m-4}{2},\frac{m-2}{2}))})=(\overbrace{1,\ldots,1}^k,0,\ldots,0),$$
the first two rows of the matrix $\varphi(T_{\alpha_0})$ are
$$
\begin{pmatrix}
\varphi(P_{(\alpha_0\times (\frac{m-4}{2},\frac{m-4}{2}+1))})  \\
 \varphi(P_{(\alpha_0\times (\frac{m-4}{2},\frac{m-4}{2}+2))})
\end{pmatrix}=\begin{pmatrix}
\overbrace{1,1,\ldots,1}^k,& 0,\ldots,0,& 0,\ldots,0 \\
1,0,\ldots,0,&\underbrace{1,\ldots,1}_{k-1},&0,\ldots,0
\end{pmatrix},
$$
the first three rows of the matrix $\varphi(T_{\alpha_0})$  are
$$
\begin{pmatrix}
\varphi(P_{(\alpha_0\times (\frac{m-4}{2},\frac{m-4}{2}+1))})  \\
 \varphi(P_{(\alpha_0\times (\frac{m-4}{2},\frac{m-4}{2}+2))})\\
  \varphi(P_{(\alpha_0\times (\frac{m-4}{2},\frac{m-4}{2}+3))})
\end{pmatrix}=\begin{pmatrix}
\overbrace{1,1,\ldots,1}^k,& 0,\ldots,0,& 0,\ldots,0, & 0,\ldots,0 \\
1,0,\ldots,0,&\underbrace{1,\ldots,1}_{k-1},&0,\ldots,0, & 0,\ldots,0\\
0,1,\ldots,0,&1,0,\ldots,0&\underbrace{1,\ldots,1}_{k-2},& 0,\ldots,0
\end{pmatrix},
$$
etcetera. It follows that
$$\varphi\big(P_{(\alpha_0\times R_{\frac{m-4}{2}})}\big)=A_k^0.$$
Now, let $I^k$ be the matrix given by the first $k$ rows of the
identity matrix $I_{C_2^{k+1}}$. Then, by an analogous argument
{\small \begin{align*} \varphi\big(P_{(\alpha_0\times
R_{\frac{m-4}{2}})}\cup P_{(\alpha_0\times
 R_{\frac{m-2}{2}})}\big)& =A_k^0(I^k)\sqcup A_{k-1}^0,\\
\varphi\big(P_{(\alpha_0\times R_{\frac{m-4}{2}})}\cup
P_{(\alpha_0\times R_{\frac{m-2}{2}})}\cup P_{(\alpha_0\times
R_{\frac{m}{2}})}\big)&=A_k^0(I^{k+(k-1)})
\sqcup A_{k-1}^0(I^{k-1})\sqcup A_{k-2}^0,\\
&\;\;\vdots\\
\varphi\big(P_{(\alpha_0\times R_{\frac{m-4}{2}})}
\cup \cdots \cup
P_{(\alpha_0\times R_{m-2})}\big)&=A_k^0(I^{k+(k-1)+\cdots+2})\sqcup
A_{k-1}^0(I^{(k-1)+\cdots+2})\\
&\qquad\qquad\qquad\qquad\qquad\sqcup\cdots\sqcup A_{2}^0(I^2)\sqcup A_1^0\\
\varphi\big(P_{(\alpha_0\times R_{\frac{m-4}{2}})}\cup \cdots \cup
P_{(\alpha_0\times R_{m-1})}\big)&=A_k^0(I^{k+(k-1)+\cdots+1})\sqcup
A_{k-1}^0(I^{(k-1)+\cdots+1})\\
 &\qquad\qquad\qquad\qquad    \sqcup\cdots\sqcup A_{2}^0(I^3)\sqcup
A_1^0(I^1)\\
&=A_k^1,
\end{align*}}where for the last equality we just notice that
$I^{k+(k-1)+\cdots+1}= I_{C_2^{k+1}}$ and similarly for the other
$I^t$ in that formula. Therefore
$$\varphi(T_{\alpha_0}) = \varphi(T_{(1,2,\ldots,(m-8)/2,(m-6)/2)})=\varphi\big(P_{(\alpha_0\times
R_{\frac{m-4}{2}})}\cup \cdots \cup P_{(\alpha_0\times
R_{m-1})}\big) =A_k^1.$$
\end{proof}

Recall now that
$$T_{(1,2,\ldots,\frac{m-8}{2},\frac{m-6}{2})}:=(P_1,\ldots,P_{(m-8)/2},
P_{(m-6)/2})\times C_2(\Sigma_{[(m-6)/2,m]}).$$
Then, define the triangular arrays:

$$\begin{array}{ll}{\begin{array}{l}
T_{1}^1 =T_{(1,2,\ldots,\frac{m-8}{2},\frac{m-6}{2})}\\
\displaystyle T_{2}^1 =\bigcup_{i=0}^{\frac{m+2}{2}}T_{(1, 2,\ldots,
\frac{m-8}{2}, \frac{m-6}{2}+i)}\\
\displaystyle T_{3}^1 =T_{2}^1\cup\bigcup_{i=0}^{\frac{m}{2}}T_{(1, 3,\ldots,
\frac{m-6}{2}, \frac{m-4}{2}+i)}\\
\displaystyle \displaystyle T_{4}^{1} =T_{3}^1\cup \bigcup_{i=0}^{\frac{m-2}{2}}T_{(1,
4,\ldots,\frac{m-4}{2},\frac{m-2}{2}+i)}\\
\;\;\,\quad \vdots  \\
\displaystyle T_{\frac{m+4}{2}}^{1} =T_{\frac{m+2}{2}}^1\cup\bigcup_{i=0}^{1}T_{(1, \frac{m+4}{2},\ldots,m-4,m-3+i)}\\
\displaystyle  T_{\frac{m+6}{2}}^{1} =T_{\frac{m+4}{2}}^1\cup T_{(1,
\frac{m+6}{2},\ldots,m-3,m-2)}
\end{array}} &
{\begin{array}{l}
T_{2}^2 =T_{(2,3,\ldots,\frac{m-6}{2},\frac{m-4}{2})}\\
\displaystyle  T_{3}^2 =\bigcup_{i=0}^{\frac{m}{2}}T_{(2,3,\ldots, \frac{m-6}{2},
    \frac{m-4}{2}+i)}\\
\displaystyle  T_{4}^2 =T_{3}^2\cup\bigcup_{i=0}^{\frac{m-2}{2}}T_{(2,4,\ldots,
\frac{m-4}{2}, \frac{m-2}{2}+i)}\\
\displaystyle T_{5}^{2} =T_{4}^2\cup
\bigcup_{i=0}^{\frac{m-4}{2}}T_{(2,5,\ldots,\frac{m-2}{2},
    \frac{m}{2}+i)}\\
\;\;\,\quad \vdots \\
\!\!\! \!\!\! \displaystyle T_{\frac{m+4}{2}}^{2} =T_{\frac{m+2}{2}}^2\cup\bigcup_{i=0}^{1}T_{(2, \frac{m+4}{2},\ldots,m-3+i)}\\
\!\!\! \!\!\! \displaystyle   T_{\frac{m+6}{2}}^{2} =T_{\frac{m+4}{2}}^2\cup T_{(2,
\frac{m+6}{2},\ldots,m-2)}
\end{array}} \\
 & \\\hline & \\
{\begin{array}{l}
T_{3}^3 = T_{(3,4,\ldots,\frac{m-4}{2},\frac{m-2}{2})}\\
\displaystyle  T_{4}^3 = \bigcup_{i=0}^{\frac{m-2}{2}}T_{(3,4,\ldots, \frac{m-4}{2},\frac{m-2}{2}+i)}\\
\displaystyle  T_{5}^3 = T_{4}^3\cup\bigcup_{i=0}^{\frac{m-4}{2}}T_{(3,5,\ldots,
\frac{m-2}{2},
    \frac{m}{2}+i)}\\
\displaystyle  T_{6}^{3} = T_{5}^3\cup
\bigcup_{i=0}^{\frac{m-6}{2}}T_{(3,6,\ldots,\frac{m}{2},
    \frac{m+2}{2}+i)}\\
\;\;\,\quad \vdots  \hspace{5.3cm} \cdots \\
\displaystyle T_{\frac{m+4}{2}}^{3} = T_{\frac{m+2}{2}}^3\cup\bigcup_{i=0}^{1}T_{(3,6,\ldots,m-3+i)}\\
T_{\frac{m+6}{2}}^{3} = T_{\frac{m+4}{2}}^3\cup T_{(3,
\frac{m+6}{2},\ldots,m-2)}
\end{array}}  &
{\begin{array}{l}
T_{\frac{m+2}{2}}^{\frac{m+2}{2}}=T_{(\frac{m+2}{2},\frac{m+4}{2},\ldots,m-4,m-2)}\\
\displaystyle  T_{\frac{m+4}{2}}^{\frac{m+2}{2}}=\bigcup_{i=0}^1T_{(\frac{m+2}{2},\frac{m+4}{2},\ldots,m-4,m-3+i)}\\
T_{\frac{m+6}{2}}^{\frac{m+2}{2}}=T_{\frac{m+4}{2}}^{\frac{m+2}{2}}\cup T_{(\frac{m+2}{2},\frac{m+6}{2},\ldots,m-4,m-2)}\\
\quad\qquad \vdots\\
T_{\frac{m+6}{2}}^{\frac{m+2}{2}}=T_{(\frac{m+4}{2},\frac{m+6}{2},\ldots,m-3,m-2)}.
\end{array}} \\
\\ \hline & \\
\end{array}  $$
Observe now that for these triangles we have that
$$T_{m}=T_{\frac{m+6}{2}}^{1}\cup T_{\frac{m+6}{2}}^{2}\cup
\cdots \cup T_{\frac{m+6}{2}}^{\frac{m+4}{2}}=C_{\frac{m-2}{2}}(\Sigma_m).$$

\begin{lem}\label{2.3}
Let $k=(m+2)/2$,  for an even integer $m$ such that $8\leq m\leq n$. Then, we have that ${\EuScript L}_k=M_m$.
\end{lem}
\begin{proof}
Recall that $C_{(m-2)/2}(\Sigma_m)=\bigcup_{\alpha\in
I((m-6)/2,m-2)}T_{\alpha}$ (see Lemma~$\ref{rem2.1}$) and the fact that
$T^1_1=T_{(1,2,\ldots,\frac{m-8}{2},\frac{m-6}{2})}$. Then, by
Lemma~$\ref{lem2.2}$ and arguments already discussed, we have
{\small
\begin{align*}
\varphi(T_{1}^1)&=A_k^0(I_{C_2^{k+1}})\sqcup A_{k-1}^0(I_{C_2^{k}})\sqcup \cdots\sqcup A_2^0(I_{C_2^{3}})\sqcup A_1^0(I_{C_2^{2}})= A_k^1,\\
\varphi(T_{2}^1)&=
A_k^1(I_{C_3^{k+2}})\sqcup A_{k-1}^1(I_{C_3^{k+1}})\sqcup \cdots\sqcup A_2^1(I_{C_3^{4}})\sqcup A_1^1(I_{C_3^{3}})= A_k^2,\\
\varphi(T_{3}^1)&=
A_k^2(I_{C_4^{k+3}})\sqcup A_{k-1}^2(I_{C_4^{k+2}})\sqcup \cdots\sqcup A_2^2(I_{C_4^{5}}) \sqcup A_1^2(I_{C_4^{4}})= A_k^3,\\
&\;\;\vdots\\
\varphi(T_{\frac{m+6}{2}}^{1})&=
A_k^{k-5}(I_{C_{\frac{m-4}{2}}^{m-2}})\sqcup A_{k-1}^{k-5}(I_{C_{\frac{m-4}{2}}^{m-1}})\sqcup \cdots\sqcup A_2^{k-5}(I_{C_{\frac{m-4}{2}}^{\frac{m-2}{2}}})\sqcup A_1^{k-5}(I_{C_{\frac{m-4}{2}}^{\frac{m-4}{2}}})= A_k^{k-4}.\\
\mathrm{And\;so}\\
\varphi\big(T_{m}^{}\big)&
=\varphi\big(T_{\frac{m+6}{2}}^{1}\cup T_{\frac{m+6}{2}}^{2}\cup \cdots \cup T_{\frac{m+6}{2}}^{\frac{m+4}{2}}\big),\\
&=A_k^{k-4}(I_{C_{\frac{m-2}{2}}^{m-1}})\sqcup A_{k-1}^{k-4}(I_{C_{\frac{m-2}{2}}^{m}})\sqcup \cdots\sqcup A_2^{k-4}(I_{C_{\frac{m-2}{2}}^{\frac{m}{2}}})\sqcup A_1^{k-4}(I_{C_{\frac{m-2}{2}}^{\frac{m-2}{2}}}),\\
&=A_k^{k-3}= {\EuScript L}_k.
\end{align*}}
 Now, since $C_{\frac{m-2}{2}}(\Sigma_m)=T_m$, then
 $$M_m=\Big(\varphi^{\beta(1),\ldots,\beta(\frac{m}{2})}_{P_{\alpha(1)},
 \ldots, P_{\alpha(\frac{m-2}{2})},P_{i}}\Big)_{\alpha\in I(\frac{m-2}{2},m),\,\beta\in
 I(\frac{m}{2},m)} = \varphi(C_{\frac{m-2}{2}}(\Sigma_m)) = {\EuScript
 L}_k,$$
 for
 $i=1,\ldots,m$.
Since $\{P_{\alpha(1)},\ldots,
P_{\alpha(\frac{m-2}{2})}:\;\alpha\in
I(\frac{m-2}{2},m)\}=C_{\frac{m-2}{2}}(\Sigma_m)$, then
$M_m={\EuScript L}_k$.
\end{proof}

\begin{rem}\label{rem2.1}
For $n$ even and $r=(n+2)/2$ consider integers $1\leq \ell\leq r-2$ and sequences of integers
$$1\leq a_1<a_2<\cdots<a_{2\ell}\leq 2n\quad\text{such that $a_i+a_j\neq 2n+1$},$$
and define
$$\Sigma_{a_1,\ldots,a_{2\ell}}:=\{P_i\in\Sigma_n: i+a_j\neq 2n+1,\, 2n-i+1+a_j\neq 2n+1\}.$$
Then:
\medskip

\noindent{(1)} We have $\big|\Sigma_{a_1,\ldots,a_{2\ell}}\big|=n-2\ell$.
\medskip

\noindent{(2)} For  $1\leq \ell \leq r-2$, letting
\begin{align*}
\Sigma\{a_1,\ldots,a_{2\ell}\}&:=(a_1,\ldots, a_{2\ell})\times C_{(n-2(\ell+1))/2}(\Sigma_{a_1,\ldots ,a_{2\ell}}) \\
&=\big\{\big(a_1\ldots a_{2\ell}, P_{\alpha(1)},\ldots, P_{\alpha((n-2(\ell+1))/2)}\big);\text{such that}\\
& \big(\alpha(1),\ldots, \alpha((n-2(\ell+1))/2)\big)\in I(
(n-2(\ell+1))/2,n-2\ell)\big\},
\end{align*}
then $\big|\Sigma\{a_1,\ldots,a_{2\ell}\}\big|=Q_{a_1\ldots
a_{2\ell}}C^{n-2\ell}_{(n-2(\ell+1))/2)}$, where
$$Q_{a_1\ldots a_{2\ell}}:=\big|\{(a_1,\ldots,a_{2\ell})\in I(2\ell,2n): a_i+a_j\neq 2n+1,\, 1\leq \ell\leq n-2\}\big|.$$

\noindent{(3)} For $\ell=0$, setting
$\Sigma\{\emptyset\}:=C_{(n-2)/2}(\Sigma_n)$, we have that
$\big|\Sigma\{\emptyset\}\big|=C^n_{(n-2)/2}$.
\end{rem}

\begin{lem}\label{lem2.3}
For any $(a_1,\ldots,a_{2k})\in I(2k,2n)$ there exists a bijection
between the set $\Sigma\{a_1,\ldots,a_{2k}\}$ and
the set of $(r-k)$-planes in ${\mathbb P}(\ker f)$
of the form $\Pi_{\alpha_{st}}$, for $\alpha_{st}\in I(n-2,2n)$.
\end{lem}
\begin{proof}
From the Pl\"ucker linear relations
 $$X_{1\aspace[1.0pt] 2n}+X_{2\aspace[1.0pt] (2n-1)}+\cdots+X_{n\aspace[1.0pt] (n+1)}=0$$
filling each of the $\aspace[1.5pt]$ with sequences
 $$a_1,a_2,\ldots,a_{2k},P_{\alpha(1)},\ldots, P_{\alpha((n-2(k+1))/2)}\in \Sigma\{a_1,\ldots,a_{2k}\}$$
 eliminates $2k+(n-2(k+1))/2=r+k-2$ variables, which gives an $(r-k)$-plane  of the $\Pi_{\alpha_{st}}$-planes, for $\alpha_{st}\in I(n-2,2n)$.
 This is a surjective function, since for any $(r-k)$-plane  of the form $\Pi_{\alpha_{st}}=\sum_{i=1}^nc_{i,\alpha_{st},2n-i+1}X_{i,\alpha_{st},2n-i+1}$,
 with $\alpha_{st}\in I(n-2,2n)$,
where
$$c_{i,\alpha_{st},2n-i+1}=\begin{cases}
1  & \text{if $|\text{supp}\{i,\alpha_{st},2n-i+1\}|=n$}, \\
0 & \text{otherwise},
\end{cases}$$
we have two cases, either $k=0$, in this case $\alpha_{st}=(P_{\alpha(1)},\ldots, P_{\alpha((n-2)/2))})\in \Sigma\{\emptyset\}$ or $k>0$, in this case $\alpha_{st}=(a_1,\ldots,a_{2k},P_{\alpha(1)},\ldots, P_{\alpha((n-2(k+1))/2))})$,
with $a_i+a_j\neq 2n+1$ and $a_i+\alpha(j)\neq 2n+1$, for all $i\neq j$.
In the second case, $\alpha_{st}\in \Sigma\{a_1,\ldots,a_{2k}\}$. Injectivity is direct.
\end{proof}
From Lemma \ref{lem2.3} it follows that, for $0\leq \ell\leq r-2$,
\begin{itemize}
\item For $\ell=0$, the number of $r$-planes of ${\mathbb P}(\ker
f)$ of the form $\Pi_{\alpha_{st}}$, for $\alpha_{st}\in
I(n-2,2n)$,  is $C^n_{\frac{n-2}{2}}$. \item For $\ell=1$, the
number of $(r-1)$-planes of ${\mathbb P}(\ker f)$ of the form
$\Pi_{\alpha_{st}}$, for $\alpha_{st}\in I(n-2,2n)$,  is
$Q_{a_1a_2}C^{n-2}_{\frac{n-4}{2}}$. \item For $\ell=2$, the
number of $(r-2)$-planes of ${\mathbb P}(\ker f)$ of the form
$\Pi_{\alpha_{st}}$, for $\alpha_{st}\in I(n-2,2n)$,  is
$Q_{a_1a_2a_3a_4}C^{n-4}_{\frac{n-6}{2}}$.
\item[$\vdots$]
\item For $\ell=r-3$, the number of $3$-planes of ${\mathbb
P}(\ker f)$ of the form $\Pi_{\alpha_{st}}$, for $\alpha_{st}\in
I(n-2,2n)$,  is $Q_{a_1a_2\ldots a_{n-4}}C^4_1$. \item For
$\ell=r-2$, the number of $2$-planes of ${\mathbb P}(\ker f)$ of
the form $\Pi_{\alpha_{st}}$, for $\alpha_{st}\in I(n-2,2n)$,  is
$Q_{a_1a_2\ldots a_{n-2}}C^2_0$. \item There are no $1$-planes of
the form $\Pi_{\alpha_{st}}$.
\end{itemize}
It follows that:
\begin{cor}\label{cor2.4}
 The total number of planes of ${\mathbb P}(\ker f)$ of the form
 $\Pi_{\alpha_{st}}$, for $\alpha_{st}\in I(n-2,2n)$,  is
$$C^{2n}_{n-2}=C^n_{\frac{n-2}{2}}+Q_{a_1a_2}C^{n-2}_{\frac{n-4}{2}}+
Q_{a_1a_2a_3a_4}C^{n-4}_{\frac{n-6}{2}}+\cdots+Q_{a_1a_2\ldots a_{n-4}}C^4_1+Q_{a_1a_2\ldots a_{n-2}}C^2_0.$$
\end{cor}
\begin{example} The total number of planes of ${\mathbb P}(\ker f) \subseteq \Lambda^6 E$
of the form $\Pi_{\alpha_{st}}$, for $\alpha_{st}\in I(4,12)$ is
$$C^{12}_{4}=C^6_{2}+ Q_{a_1a_2}C^{4}_{1}+
Q_{a_1a_2a_3a_4}C^{2}_{2},$$ where it is easy to see that
$$C^{12}_{4}=C^6_{2}+ 60 C^{4}_{1}+
240 C^{2}_{2},$$ since $Q_{a_1 a_2}= C_{2}^{12} - |\{P_i\}_{i=1}^6|
= 60$, and $Q_{a_1a_2a_3} = |\{(a_1,a_2,a_3) \in I(3,12): a_i + a_j
\neq 13, \text{for all} \; i, j)\}|= 240$.
\end{example}

\section{An explicit description of ${\mathbb P}(\ker f)$}\label{SecPrin}

In this section we prove the main result of this paper, that is, an
explicit description of the linear polynomials that cut out the
Lagrangian-Grassmannian $L(n,2n)$ in the Grassmann variety
$G(n,2n)$. We begin with
\begin{equation}\label{lem2.5}
I(n-2,2n)=C_{(n-2)/2}(\Sigma_n)\cup\Big(\bigcup_{\ell=1}^{r-2}
\bigcup_{(a_1,\ldots, a_{2\ell})\in I(2\ell,2n)\atop a_i+a_j\neq
2n+1}\Sigma\{a_1,\ldots, a_{2\ell}\}\Big).
\end{equation}

\begin{thm}\label{thmPeven}
For $n\geq 4$ even,  $r=(n+2)/2$ and $1\leq k\leq r-2$. As in Section \ref{sec:sec2}, let $B$ be the matrix associated to the kernel of the contraction map, that is, the linear sections
that define the Lagrangian-Grassmannian $L(n,2n)$.
Then, $B$ is the direct sum
$$B={\EuScript L}_r\oplus\Big( \bigoplus_{k=1}^{r-2}\Big(\bigoplus_{1\leq a_1<\cdots<a_{2k}\leq 2n\atop a_i+a_j\neq 2n+1}{\EuScript L}_{k+1}^{(a_1,\cdots, a_{2k})}\Big)\Big),$$
where ${\EuScript L}_{k+1}^{(a_1,\cdots, a_{2k})} \cong {\EuScript
L}_{k+1}$ (are equivalent) for $1 \leq k \leq r-2$.
\end{thm}
\begin{proof}
First, we extend the function $\varphi$ defined in
Section~\ref{Sec-Mm} to $\varphi:I(n-2,2n)\rightarrow
F^{C_n^{2n}}$ as
$$(\alpha(1),\ldots,\alpha(n-2))\mapsto \Big( \varphi^{\beta(1),\ldots,\beta(n)}_{i,\alpha(1),\ldots, \alpha(n-2),2n-i+1}\Big)_{\beta\in I(n,2n),\, 1\leq i\leq n},$$
where {\small $$
\varphi^{\beta(1),\ldots,\beta(n)}_{i,\alpha(1),\ldots,
\alpha(n-2),2n-i+1}=\begin{cases}
1  & \text{if $\{i,\alpha(1),\ldots, \alpha(n-2),2n-i+1\}=\{\beta(1),\ldots,\beta(n)\}$}, \\
0& \text{otherwise}.
\end{cases}$$}
The corresponding $C^{2n}_{n-2}\times C^{2n}_n$ matrix is
$$B=\Big(\varphi^{\beta(1),\ldots,\beta(n)}_{i,\alpha(1),\ldots, \alpha(n-2),2n-i+1}\Big)_{\alpha\in I(n-2,2n),\beta\in I(n,2n)}, 1\leq i \leq n.$$
Then, by the equality \eqref{lem2.5}
 {\small\begin{align*}
I(n-2,2n)&=
C_{(n-2)/2}(\Sigma_n)\cup\Big(\bigcup_{k=1}^{r-2}\bigcup_{(a_1,\ldots,
a_{2k})\in I(2k,2n)\atop a_i+a_j\neq 2n+1}\Sigma\{a_1,\ldots,
a_{2k}\}\Big),
\end{align*}}
and using Lemma \ref{2.3}, we obtain that 
{\small
\begin{align*}
\Big(\varphi^{\beta(1),\ldots,\beta(n)}_{i,\alpha(1),\ldots,
\alpha(n-2),
2n-i+1}\Big)_{\small {\alpha\in I(n-2,2n)}\atop \beta\in I(n,2n)}&=\Big(\varphi^{\beta(1),\ldots,\beta(n)}_{i,\alpha(1),\ldots, \alpha(n-2),2n-i+1}\Big)_{\alpha\in C_{\frac{n-2}{2}}(\sum_n)\atop \beta\in I(n,2n)}\\
&\quad \oplus\Big(\varphi^{\beta(1),\ldots,\beta(n)}_{i,\alpha(1),\ldots, \alpha(n-2),2n-i+1}\Big)_{\alpha\in \displaystyle\bigcup_{k=1}^{r-2}\bigcup_{{(a_1,\ldots,a_{2k})\in I(2k,2n)}\atop a_i + a_j \neq 2n+1}  \Sigma\{a_1,\ldots, a_{2k}\}\atop \beta\in I(n,2n)}\\
& = (\varphi(C_{\frac{n-2}{2}}(\Sigma_n))) \oplus \Big( \bigoplus_{k-1}^{r-2} \bigoplus_{(a_1,\ldots,a_{2k})\in I(2k,2n) \atop a_i+a_j\neq 2n+1} (\varphi(\Sigma(a_1,\ldots,\alpha_{2k})))\Big)\\
&={\EuScript L}_r\oplus\Big( \bigoplus_{k=1}^{r-2}\Big( \bigoplus_{1\leq a_1<\cdots<a_{2k}\leq 2n\atop a_i+a_j\neq 2n+1}{\EuScript L}_{k+1}^{(a_1,\cdots, a_{2k})}\Big)\Big).\\
\end{align*}}
\end{proof}

\subsection{When $n$ is odd}
The odd case is obtained by modifications to the even case. Basically, we just modify  Remark~\ref{rem2.1}, that is

\begin{rem}\label{rem2.odd}
For $n\geq 5$ odd and $r=(n+1)/2$ consider integers $1\leq \ell\leq
r-2$ and sequences of integers
$$1\leq a_1<a_2<\cdots<a_{2\ell +1}\leq 2n\quad\text{such that $a_i+a_j\neq 2n+1$},$$
and define
$$\Sigma_{a_1,\ldots,a_{2\ell +1}}:=\{P_i\in\Sigma_n: i+a_j\neq 2n+1,\, 2n-i+1+a_j\neq 2n+1\}.$$
Then:
\medskip

\noindent{(1)} We have $\big|\Sigma_{a_1,\ldots,a_{2\ell + 1}}\big|=n-(2\ell +1)$.
\medskip

\noindent{(2)}  For $1\leq \ell \leq r-2$, let
\begin{align*}
\Sigma\{a_1,\ldots,a_{2\ell + 1}\}&:=(a_1,\ldots, a_{2\ell + 1})\times C_{(n-(2\ell+3))/2}(\Sigma_{a_1,\ldots ,a_{2\ell +1}}) \\
&=\big\{\big(a_1\ldots a_{2\ell +1}, P_{\alpha(1)},\ldots, P_{\alpha((n-(2\ell+3))/2)}\big)\;\text{such that}\\
&\big(\alpha(1),\ldots, \alpha((n-(2\ell+3))/2)\big) \in I(\textstyle\frac{
n-(2\ell+3)}{2},n-(2\ell +1))\big\}.
\end{align*}

\noindent{(3)}  With the above notation $(\ref{lem2.5})$ becomes:
\begin{equation*}
I(n-2,2n)= \bigcup_{i=1}^n \left[(i)\times C_{(n-3)/2}(\Sigma(i))
\right]\cup \bigcup_{\ell=0}^{r-3}\Big( \bigcup_{1\leq
a_1<\ldots< a_{2\ell +1} \leq 2n \atop a_i+a_j\neq 2n+1}
\Sigma\{a_1, ,\ldots, a_{2\ell+1}\} \Big).
\end{equation*}
\end{rem}

\begin{thm}\label{thmPodd}
For an odd integer $n\geq 5$  and $r=(n+1)/2$. As in Section \ref{sec:sec2}, let $B$ be the matrix associated to the kernel of the contraction map, that is, the linear sections
that define the Lagrangian-Grassmannian $L(n,2n)$.
Then, $B$ is the direct sum
$$B={\EuScript L}_r^{n}\oplus\Big(\bigoplus_{k=0}^{r-3}
\Big(\bigoplus_{1\leq a_1< a_2< \cdots < a_{2k+1}\leq 2n \atop a_i
+ a_j \neq 2n+1} {\EuScript
L}_{k+2}^{(a_1,a_2,\dots,a_{2k+1})}\Big)\Big),$$ where ${\EuScript
L}_{k+2}^{(a_1, a_2,\dots, a_{2k+1})}\cong {\EuScript L}_{k+2}$
for $0\leq k\leq r-3$.
\end{thm}

\begin{proof}
As before, let $\Sigma_n=\{P_1,\dots,P_n\}$ and
$\Sigma(i)=\Sigma_s-\{P_i\}$ for all $i \in \{1,\ldots, n\}$ and
$1\leq s \leq n$. From Lemma~\ref{lem2.2} and \ref{2.3} it is easy
to see that the image of $(i)\times
C_{\frac{n-3}{2}}\big(\Sigma(i)\big)$ under $\varphi$ is ${\EuScript
L}_r$, i.e., $$\varphi \Big((i)\times
C_{\frac{n-3}{2}}\big(\Sigma(i)\big)\Big)={\EuScript L}_r^{i}\cong
{\EuScript L}_r \mathrm{\;for\; all\;} i=1,\dots, n.$$ Now, from
Remark~\ref{rem2.odd} and Theorem~\ref{thmPeven}, we obtain {\small
\begin{gather*}\label{eqfin} I(n-2,2n)=
\bigcup_{i=1}^{n}\left[ (i)\times
C_{\frac{n-3}{2}}\big(\Sigma(i)\big)\right] \cup
\bigcup_{k=0}^{r-3}\Big(\bigcup_{1\leq a_1 < a_2 < \dots < a_{2k+1}
\leq 2n \atop a_i + a_j \neq
2n+1}\Sigma\{a_1,a_2,\dots,a_{2k+1}\}\Big)\end{gather*}}
where
$\Sigma\{a_1,a_2,\dots,a_{2k+1}\}=(a_1,a_2,\dots,a_{2k+1})\times
C_{\frac{n-(2k+3)}{2}}\big(\Sigma_{a_1,a_2,\dots,a_{2k+1}}\big)$.
Then {\small \begin{align*} B=\varphi(I(n-2,2n)) &=
\bigoplus_{i=1}^n \varphi \Big((i)\times C_{\frac{n+1}{2}}(\Sigma
(i))\Big) \oplus \bigoplus_{k=0}^{r-3}\Big( \bigoplus_{1\leq a_1< a_2<
\ldots < a_{2k+1}\leq 2n \atop a_i + a_j \neq 2n +1}
\varphi(\Sigma( a_1, a_2,\ldots, a_{2k+1}))\Big) \\
&=\bigoplus_{i=1}^n {\EuScript L}_r^i \oplus \Big(\bigoplus_{k=0}^{r-3} \Big(\bigoplus_{1\leq a_1 < \ldots < a_{2k+1} \leq a_{2n} \atop a_i + a_j \neq 2n+1}  {\EuScript L}_{k+2}^{(a_1, a_2,\ldots, a_{2k+1})}\Big)\Big) \\
 \end{align*}}
\end{proof}

\section{The rank of the matrices $B$ and ${\EuScript L}_k$}\label{sec4}

\begin{prop}
For any field $F$ of characteristic different from $2$ and  $n$
even, Let $m= 2(k-1)$ with $2\leq k\leq \frac{n+2}{2}=r$. Suppose
that $\text{\rm rank}\, B =C^{2n}_{n-2}$. Then $\text{\rm rank}\,
{\EuScript L}_k = C^m_{\frac{m-2}{2}}$.
\end{prop}
\begin{proof}
By induction on $k\geq 3$, the case $k=3$ is Example 4
in~\cite{1}, if char $F \neq 2$. Suppose that $\text{rank}\,{\EuScript L}_k =
C_{\frac{m-2}{2}}^m$ for any $3 < k < \frac{n+2}{2}=r$ and $m=
2(k-1)$. By Theorem~\ref{thmPeven}
\begin{align}\label{exAu}
\text{\rm rank}\, B&=\text{\rm rank}\, {\EuScript L}_r + \Big(
\sum_{k=1}^{r-2} \Big( \sum_{1\leq a_1 < \cdots < a_{2k}\leq 2n
\atop a_i + a_j \neq 2n+1} \text{\rm rank}\, {\EuScript
L}_{k+1}^{(a_1,\ldots,a_{2k})} \Big)\Big) \nonumber \\ \text{\rm
rank}\, B&=\text{\rm rank}\, {\EuScript L}_r +
Q_{a_1a_2}C_{\frac{n-4}{2}}^{n-2} + Q_{a_1a_2a_3a_4}
C_{\frac{n-6}{2}}^{n-4} + \cdots + Q_{a_1a_2\cdots a_{n-2}} C_0^2
\end{align}
where the last argument is a consequence of the induction
hypothesis. From the Corollary \ref{cor2.4}, we obtain that {\small
\begin{gather}\label{exC} {\small C_{n-2}^{2n} = C_{\frac{n-2}{2}}^n +
Q_{a_1a_2}C_{\frac{n-4}{2}}^{n-2} + Q_{a_1a_2a_3a_4}
C_{\frac{n-6}{2}}^{n-4} + \cdots  + Q_{a_1a_2\ldots a_{n-2}}
C_0^2}.
\end{gather}}
Using $(\ref{exAu})$ and $(\ref{exC})$ we obtain $\text{\rm
rank}\, {\EuScript L}_r = C_{\frac{n-2}{2}}^n$.
\end{proof}

The following examples  clarify some results of Section \ref{SecPrin}.

\begin{example}
For the Lagrangian-Grassmannian $L(6,12)=G(6,12)\cap{\mathbb P}(\ker
f)$, let $P_i=(i,12-i+1)$ and $\Sigma_6=\{P_1,P_2,\ldots,P_6\}$, we
have $$ I(2,4)= C_2(\Sigma_6) \cup \big( \bigcup_{1\leq a_1< a_2
\leq 12 \atop a_1 +a_2 \neq 13}\Sigma \{a_1a_2\}\big) \cup
\big(\bigcup_{1 \leq a_1 < a_2 < a_3 < a_4 \leq 12 \atop a_i +a_j
\neq 13} \Sigma\{a_1a_2a_3a_4\} \big).$$ 
Moreover, the number of
$4$-planes is $15$, the number of $3$-planes is $240$, and the
number of $2$-planes is $240$. Then, using Theorem
\ref{thmPeven} we have
\begin{align*}
B  &={\EuScript L}_4 \oplus \Big( \bigoplus_{1\leq a_1 < a_2 \leq
12 \atop a_1 + a_2 \neq 13 } {\EuScript L}_3^{(a_1, a_2)} \Big)
\oplus \Big(\bigoplus_{1\leq a_1< a_2 <a_3 < a_4 \leq 12 \atop a_i
+ a_j \neq 13 }
{\EuScript L}_{2}^{(a_1, a_2, a_3, a_4)} \Big)\\
&=
\begin{pmatrix}
\begin{tabular}{|c c c c c c c}
\cline{1-1}\multicolumn{1}{|c|}{${\EuScript L}_4$}  & & & & \multicolumn{1}{c}{ } &    \\
\cline{2-1}\cline{1-1} \multicolumn{1}{c}{}  &
\multicolumn{1}{|c|}{${\EuScript L}_3$} & & &
\multicolumn{1}{c}{$0$}& \\ 
\cline{2-2}\multicolumn{1}{c}{} & \multicolumn{1}{c}{} &
\multicolumn{1}{c}{$ \ddots$} & & \multicolumn{1}{c}{ }  & \\ 
\cline{4-3}\multicolumn{1}{c}{ } &\multicolumn{1}{c}{$0$ } &
\multicolumn{1}{c|}{ }&\multicolumn{1}{c|}{${\EuScript L}_3$}  & & \\
\cline{5-3}\cline{4-4} \multicolumn{1}{c}{} & \multicolumn{1}{c}{}
& \multicolumn{1}{c}{}& \multicolumn{1}{c|}{} &
\multicolumn{1}{c|}{${\EuScript L}_2$}   & \\
\cline{5-5} \multicolumn{1}{c}{} & \multicolumn{1}{c}{} &
\multicolumn{1}{c}{}& \multicolumn{1}{c}{} & \multicolumn{1}{c}{}
& \multicolumn{1}{c}{$\ddots$}
\\ \cline{7-7}
\multicolumn{1}{c}{} & \multicolumn{1}{c}{} &
\multicolumn{1}{c}{}& \multicolumn{1}{c}{ } &
\multicolumn{1}{c}{ } & \multicolumn{1}{c|}{ } &   \multicolumn{1}{c|}{${\EuScript L}_2$} \\
\cline{7-7}
\end{tabular} \end{pmatrix},
\end{align*}
where there are 60 submatrices ${\EuScript L}_3$, and 240
submatrices ${\EuScript L}_2$.

On the other hand, we obtain
\begin{align*}
\text{\rm rank}\, B & =\text{\rm rank}\, {\EuScript L}_4 +
(C_2^{12} - 6)\, \text{\rm rank}\, {\EuScript L}_{3} + 240 \,\text{\rm
rank}\, {\EuScript L}_2
\\ & = \text{\rm rank}\, {\EuScript L}_4 + 60\,\text{\rm rank}\,
{\EuScript L}_3 + 240\, \text{\rm rank}\, {\EuScript L}_2.
\end{align*}
Using a direct calculation, it is easy to see that
\begin{center}
\begin{tabular}{|c|c|c|}
\cline{1-3}$\text{Char}(F)$ & $\text{rank} \,{\EuScript L}_4$ & $\text{rank} \,B$ \\
\cline{1-3} 0 & 15 & 495 \\
\cline{1-3} 2 & 10  & 430 \\
\cline{1-3}3  & 14 & 494 \\
\cline{1-3} $ p \geq 5$   & 15 & 495  \\
\cline{1-3}
\end{tabular}
\end{center}
\end{example}

\begin{example}
For the Lagrangian-Grassmannian $L(5,10)=G(5,10)\cap{\mathbb P}(\ker
f)$. Using the notation of  Theorem \ref{thmPodd},  let
$P_i=(i,12-i+1)$, $\Sigma_5 = \{P_1,P_2,P_3,P_4,P_5\}$ for $i\in
\{1,2,\ldots,5 \}$, and $(i)\times C_3(\Sigma(i))= \{iP_j : j\neq i,
j\in \{1,2,\ldots,5\}\}$. Since
$$ I(3,10)= \Big(\bigcup_{i=1}^5 (i)\times C_3(\Sigma(i))\Big) \cup \Big( \bigcup_{1\leq a_1< a_2< a_3 \leq 10 \atop
a_i +a_j \neq 11}\Sigma \{a_1a_2a_3\}\Big).$$ Then, using 
Theorem \ref{thmPodd} we have
\begin{gather*}
B = \varphi(I(3,10))= \bigoplus_{i=1}^5 {\EuScript L}_3^{\{i\}}
\oplus \Big( \bigoplus_{1\leq a_1 < a_2< a_3 \leq 10 \atop a_i +
a_j \neq 11 } {\EuScript L}_2^{(a_1, a_2, a_3)} \Big).
\end{gather*}
Then
\begin{gather*}
B= \begin{pmatrix}
\begin{tabular}{|c c c c c c c}
\cline{2-1} \multicolumn{1}{c}{}  &
\multicolumn{1}{|c|}{${\EuScript L}_3$} & & &
\multicolumn{1}{c}{$0$}& \\
\cline{2-2}\multicolumn{1}{c}{} & \multicolumn{1}{c}{} &
\multicolumn{1}{c}{$ \ddots$} & & \multicolumn{1}{c}{ }  & \\
\cline{4-3}\multicolumn{1}{c}{ } &\multicolumn{1}{c}{$0$ } &
\multicolumn{1}{c|}{ }&\multicolumn{1}{c|}{${\EuScript L}_3$}  & & \\
\cline{5-3}\cline{4-4} \multicolumn{1}{c}{} & \multicolumn{1}{c}{}
& \multicolumn{1}{c}{}& \multicolumn{1}{c|}{} &
\multicolumn{1}{c|}{${\EuScript L}_2$}   & \\
\cline{5-5} \multicolumn{1}{c}{} & \multicolumn{1}{c}{} &
\multicolumn{1}{c}{}& \multicolumn{1}{c}{} & \multicolumn{1}{c}{}
& \multicolumn{1}{c}{$\ddots$}
\\ \cline{7-7}
\multicolumn{1}{c}{} & \multicolumn{1}{c}{} &
\multicolumn{1}{c}{}& \multicolumn{1}{c}{ } &
\multicolumn{1}{c}{ } & \multicolumn{1}{c|}{ } &   \multicolumn{1}{c|}{${\EuScript L}_2$} \\
\cline{7-7}
\end{tabular} \end{pmatrix},
\end{gather*}
where there are 5 submatrices ${\EuScript L}_3$, and 7 submatrices
${\EuScript L}_2$.

Using a direct calculation, it is easy to see that
\begin{center}
\begin{tabular}{|c|c|c|}
\cline{1-3}$\text{Char}(F)$ & $\text{rank}\,{\EuScript L}_3$ & $\text{rank} \,B$ \\
\cline{1-3} 0 & 4 & 27 \\
\cline{1-3} 2 & 3  & 22 \\
\cline{1-3} $ p \geq 3 $   & 4 & 27  \\
\cline{1-3}
\end{tabular}
\end{center}

\end{example}

\begin{example}
For the Lagrangian-Grassmannian $L(7,14)=G(7,14)\cap{\mathbb P}(\ker
f)$. Using the notation of the Theorem \ref{thmPodd}. For $i\in
\{1,2,\ldots,7\}$, let $P_i=(i,14-i+1)$, $\Sigma_7 =
\{P_1,P_2,P_3,P_4,P_5,P_6,P_7\}$, and $$(i)\times C_4(\Sigma(i))=
\{i P_j P_t : i\neq j, i\neq t,\mathrm{\; with\;} i,j,t\in
\{1,2,\ldots,7\}\}.$$ Since
$$ I(5,14)= \bigcup_{i=1}^7 \big((i)\times C_4(\Sigma (i))\big) \cup \big( \bigcup_{1\leq a_1< a_2< a_3 \leq 14 \atop
a_i +a_j \neq 15}\Sigma \{a_1a_2a_3\}\big) \cup \big( \bigcup_{1\leq
a_1< a_2< \ldots < a_5 \leq 14 \atop a_i +a_j \neq 15}\Sigma
\{a_1a_2a_3a_4 a_5\}\big).$$ 
Then, using  Theorem \ref{thmPodd} we have
\begin{align*}
B & = \varphi(I(5,14)))= \bigoplus_{i=1}^7 \varphi((i)\times
C_4(\Sigma(i))) \oplus \bigoplus_{1\leq a_1 < a_2< a_3 \leq 14 \atop
a_i + a_j \neq 15} \varphi(\Sigma\{a_1a_2a_3\}) \\
&\qquad\qquad \qquad \qquad \qquad \qquad \qquad  \oplus
\bigoplus_{1\leq a_1 < a_2 <\ldots < a_5 \leq 14 \atop a_i +
a_j \neq 15 } \varphi(\Sigma\{a_1a_2a_3a_4a_5 \})  \\
& = \bigoplus_{i=1}^7 {\EuScript L}_4^{\{i\}} \oplus \Big(
\bigoplus_{1\leq a_1 < a_2< a_3 \leq 14 \atop a_i + a_j \neq 15 }
{\EuScript L}_3^{(a_1 a_2 a_3)} \Big) \oplus \Big( \bigoplus_{1\leq
a_1< a_2<\ldots <a_4<a_5\leq 14 \atop a_i + a_j \neq 15} {\EuScript
L}_2^{(a_1a_2 a_3a_4 a_5)} \Big),
\end{align*}
where ${\EuScript L}_4^{\{i\}} = {\EuScript L}_4$, ${\EuScript
L}_4^{(a_1a_2a_3)} = {\EuScript L}_3$ and ${\EuScript
L}_4^{(a_1a_2a_3a_4a_5)} = {\EuScript L}_2$.  Then
\begin{gather*}
B= \begin{pmatrix}
\begin{tabular}{|c c c c c c c c c c}
\cline{2-1} \multicolumn{1}{c}{}  &
\multicolumn{1}{|c|}{${\EuScript L}_4$} & & \multicolumn{1}{c}{} &
& & \multicolumn{1}{c}{$0$}& \\ 
\cline{2-2}\multicolumn{1}{c}{} & \multicolumn{1}{c}{} &
\multicolumn{1}{c}{$ \ddots$} & & \multicolumn{1}{c}{ }  & \\
\cline{4-3}\multicolumn{1}{c}{ } &\multicolumn{1}{c}{ } &
\multicolumn{1}{c|}{ }&\multicolumn{1}{c|}{${\EuScript L}_4$}  & & \\
\cline{5-3}\cline{4-4} \multicolumn{1}{c}{} & \multicolumn{1}{c}{}
& \multicolumn{1}{c}{}& \multicolumn{1}{c|}{} &
\multicolumn{1}{c|}{${\EuScript L}_3$}   & \\
\cline{5-5} \multicolumn{1}{c}{} & \multicolumn{1}{c}{$0$} &
\multicolumn{1}{c}{}& \multicolumn{1}{c}{} & \multicolumn{1}{c}{}
& \multicolumn{1}{c}{$\ddots$}
\\ \cline{7-7}
\multicolumn{1}{c}{} & \multicolumn{1}{c}{} &
\multicolumn{1}{c}{}& \multicolumn{1}{c}{ } &
\multicolumn{1}{c}{ } & \multicolumn{1}{c|}{ } &   \multicolumn{1}{c|}{${\EuScript L}_3$} \\
\cline{7-7}\cline{8-8} \multicolumn{1}{c}{} & \multicolumn{1}{c}{}
& \multicolumn{1}{c}{} & \multicolumn{1}{c}{}& \multicolumn{1}{c}{
} &
\multicolumn{1}{c}{ } & \multicolumn{1}{c|}{ } &   \multicolumn{1}{c|}{${\EuScript L}_2$} \\
\cline{8-8}  \multicolumn{1}{c}{} & \multicolumn{1}{c}{} &
\multicolumn{1}{c}{} & \multicolumn{1}{c}{} &
\multicolumn{1}{c}{}& \multicolumn{1}{c}{} & \multicolumn{1}{c}{}
&\multicolumn{1}{c}{} & \multicolumn{1}{c}{$\ddots$} \\
\cline{10-10}  \multicolumn{1}{c}{} & \multicolumn{1}{c}{} &
\multicolumn{1}{c}{} & \multicolumn{1}{c}{} &
\multicolumn{1}{c}{}& \multicolumn{1}{c}{} & \multicolumn{1}{c}{}
&\multicolumn{1}{c}{} & \multicolumn{1}{c|}{} &
\multicolumn{1}{c|}{${\EuScript L}_2$} \\ \cline{10-10}
\end{tabular} \end{pmatrix}
\end{gather*}
where there are 7 submatrices ${\EuScript L}_4$, $301$ submatrices
${\EuScript L}_3$ and $693$ submatrices ${\EuScript L}_2$.

Using a direct calculation, it is easy to see that
\begin{center}
\begin{tabular}{|c|c|c|}
\cline{1-3}$\text{Char}(F)$ & $\text{rank}\,{\EuScript L}_3$ & $\text{rank}\, B$ \\
\cline{1-3} 0 & 15 & 2002 \\
\cline{1-3} 2 & 10  & 1666 \\
\cline{1-3}3  & 14 & 1995 \\
\cline{1-3} $ p \geq 5$   & 15 & 2002  \\
\cline{1-3}
\end{tabular}
\end{center}
\end{example}


\end{document}